\documentclass{elsarticle}

\usepackage{amssymb}
\usepackage{amsmath}
\usepackage{amsthm}
\newtheorem{theorem}{Theorem}
\newtheorem{lemma}{Lemma}
\newtheorem*{namedtheorem}{Theorem EKR \cite{TheOriginalEKR}}

\begin{document}

\title{An Erd\H{o}s-Ko-Rado theorem for subset partitions}

 \author{Adam Dyck}\corref{cor1}
 \ead{dyck204a@uregina.ca}

 \author{Karen Meagher \fnref{fn1}}
 \ead{karen.meagher@uregina.ca}
 \cortext[cor1]{Corresponding author}
 \fntext[fn1]{Research supported by NSERC.}

 \address{Department of Mathematics and Statistics,\\
 University of Regina, 3737 Wascana Parkway, S4S 0A4 Regina SK, Canada}


\begin{abstract}
  A $k\ell$-subset partition, or \emph{$(k,\ell)$-subpartition}, is a
  $k\ell$-subset of an $n$-set that is partitioned
  into $\ell$ distinct classes, each of size $k$. Two
  $(k,\ell)$-subpartitions are said to \emph{$t$-intersect} if they
  have at least $t$ classes in common. In this paper, we prove an
  Erd\H{o}s-Ko-Rado theorem for intersecting families of
  $(k,\ell)$-subpartitions. We show that for $n \geq k\ell$,  $\ell \geq
  2$ and $k \geq 3$, the largest $1$-intersecting family contains at most
  $\frac{1}{(\ell-1)!}\binom{n-k}{k}\binom{n-2k}{k}\cdots\binom{n-(\ell-1)k}{k}$
  $(k,\ell)$-subpartitions, and that this bound is only attained by
  the family of $(k,\ell)$-subpartitions with a common fixed class,
  known as the \emph{canonical intersecting family of
    $(k,\ell)$-subpartitions}. Further, provided that $n$ is
  sufficiently large relative to $k,\ell$ and $t$, the largest
  $t$-intersecting family is the family of $(k,\ell)$-subpartitions
  that contain a common set of $t$ fixed classes.
\end{abstract}

\maketitle

\section{Introduction}

In this paper, we shall prove an Erd\H{o}s-Ko-Rado theorem for
intersecting families of subset partitions.  The EKR theorem gives the
size and structure of the largest family of intersecting sets, all of
the same size from a base set. This theorem has an interesting
history, Erd\H{o}s claims in~\cite{MR905276} that the work was done in
1938, but due to lack of interest in combinatorics at the time, is
wasn't until 1961 that the paper was published. Once the result did
appear in the literature it sparked a great deal of interest in
extremal set theory.

To start, we must consider some relevant notation and background
information.  For any positive integer $n$, denote
$[n]:=\{1,\dots,n\}$. A \emph{$k$-set} is a subset of size $k$ from
$[n]$. Two $k$-sets $A$ and $B$ are said to \emph{intersect} if $|A
\cap B| \geq 1$, and for $1 \leq t \leq k$, they are said to be
\emph{$t$-intersecting} if $|A \cap B| \geq t$. A \emph{canonical
  $t$-intersecting family of $k$-sets} is one that contains all
$k$-sets with $t$ fixed elements.

\begin{namedtheorem}
  Let $n \geq k \geq t \geq 1$, and let $\mathcal{F}$ be a
  $t$-intersecting family of $k$-sets from $[n]$. If $n$ is
  sufficiently large compared to $k$ and $t$, then $|\mathcal{F}| \leq
  \binom{n-t}{k-t}$; further, equality holds if and only if
  $\mathcal{F}$ is a canonical $t$-intersecting family of $k$-sets.
\end{namedtheorem}

The exact bound on $n$ is know to be $n \geq (t+1)(k-t+1)$ (an elegant
proof of this that using algebriac graph theory is given by Wilson
in~\cite{MR0771733}). If $n$ is smaller than this bound, then there are $t$-intersecting
families that are larger than the canonical $t$-intersecting family. A
complete characterization of the families of maximum size for all
values of $n$ is given by Ahlswede and Khachatrian in~\cite{MR1429238}.

From here, many EKR-type theorems have been developed by incorporating
other combinatorial objects. Frankl and Wilson have considered this
theorem for vector spaces over a finite field~\cite{EKRvspace}, Rands
for blocks in a design~\cite{EKRblock}, Cameron and Ku for
permutations~\cite{EKRperm}, Ku and Leader for partial
permutations~\cite{EKRpperm}, Brunk and Huczynska for
injections~\cite{EKRinj}, and Ku and Renshaw for set partitions and
cycle-intersecting permutations~\cite{EKRpart}. All of these cases
consider combinatorial objects that are made up of what we shall call
\emph{atoms}, and two objects intersect if they contain a common atom
and $t$-intersect if the contain $t$ common atoms. To say that ``an
EKR-type theorem holds'' means that the largest set of intersecting (or
$t$-intersecting) objects is the set of all objects that contain a
common atom (or a common $t$-set of atoms).

In this paper, we shall prove that an EKR-type theorem holds for an
object which we call a \emph{subset partition}. We begin by outlining
the appropriate notation.

A \emph{uniform $\ell$-partition} of $[n]$ is a division of $[n]$ into
$\ell$ distinct, non-empty subsets, known as \emph{classes}, where
each class has the same size and the union of these classes is $[n]$.
Further, a \emph{uniform $k\ell$-subset partition} $P$ is a uniform
$\ell$-partition of a subset of $k\ell$ elements from $[n]$. We shall
also call $P$ a \emph{$(k,\ell)$-subpartition}.  If $P$ is a
$(k,\ell)$-subpartition of $[n]$, then $P=\{P_{1},\dots,P_{\ell}\}$ and
$|P_{i}|=k$ for $i \in \{1,\dots,\ell\}$, with
$|\bigcup\limits_{i=1}^{\ell}P_i| = k\ell$.  Let $U_{\ell,k}^{n}$
denote the set of all $(k,\ell)$-subpartitions from $[n]$, and define
\[
U(n,\ell,k) := |U_{\ell,k}^{n}| =
\frac{1}{\ell!} \binom{n}{k}\binom{n-k}{k}\cdots\binom{n-(\ell-1)k}{k}
= \frac{1}{\ell!}\prod\limits_{i=0}^{\ell-1}\binom{n-ik}{k}.
\]

Two $(k,\ell)$-subpartitions $P=\{P_{1},\dots,P_{\ell}\}$ and
$Q=\{Q_{1},\dots,Q_{\ell}\}$ are said to be \emph{intersecting} if
$P_{i}=Q_{j}$ for some $i,j \in \{1,\dots,\ell\}$. Further, for $1
\leq t \leq \ell$, $P$ and $Q$ are said to be
\emph{$t$-intersecting} if there is an ordering of the classes such that $P_i = Q_i$ for $i = 1,\dots,t$.

A \emph{canonical $t$-intersecting family of $(k,\ell)$-subpartitions} is
a family that contains every $(k,\ell)$-subpartition with a fixed
set of $t$ classes. Such a family has size 
\begin{equation}
  \label{canont} \tag{*} U(n-tk,\ell-t,k) =
  \frac{1}{(\ell-t)!}\prod\limits_{i=t}^{\ell-1}\binom{n-ik}{k}.
\end{equation}
In particular, a  \emph{canonical intersecting family of $(k,\ell)$-subpartitions} has size 
\begin{equation}
  \label{canon1} \tag{**} U(n-k,\ell-1,k) =
  \frac{1}{(\ell-1)!}\prod\limits_{i=1}^{\ell-1}\binom{n-ik}{k}.
\end{equation}
Finally, note that
\begin{equation} \label{recursive} \tag{\dag}
U(n,\ell,k) = \frac{1}{\ell}\binom{n}{k}U(n-k,\ell-1,k),
\end{equation}
and $U(n,0,0) = 1$ for $n \geq 0$.

We shall not consider the cases when $k=1$, as this reduces to the
original EKR theorem \cite{TheOriginalEKR}, when $\ell=1$, where
intersection is trivial, or when $t=\ell$, where intersection is also
trivial.


\begin{theorem}\label{one}
  Let $n, k, \ell$ be positive integers with $n \geq k\ell$, $\ell
  \geq 2$, and $k \geq 3$. If $\mathcal{P}$ is an intersecting family
  of $(k,\ell)$-subpartitions, then
$$|\mathcal{P}| \leq \frac{1}{(\ell-1)!} \prod\limits_{i=1}^{\ell-1}\binom{n-ik}{k}.$$
Moreover, this bound can only be attained by a canonical intersecting family of $(k,\ell)$-subpartitions.
\end{theorem}

\begin{theorem}\label{two}
  Let $n, k, \ell, t$ be positive integers with $n \geq
  n_0(k,\ell,t)$ and $1 \leq t \leq \ell-1$. If
  $\mathcal{P}$ is a $t$-intersecting family of
  $(k,\ell)$-subpartitions, then
$$|\mathcal{P}| \leq \frac{1}{(\ell-t)!} \prod\limits_{i=t}^{\ell-1}\binom{n-ik}{k}.$$
Moreover, this bound can only be attained by a canonical $t$-intersecting family of $(k,\ell)$-subpartitions.
\end{theorem}

In their 2005 paper, Meagher and Moura \cite{MnM} introduced
Erd\H{o}s-Ko-Rado theorems for $t$-intersecting partitions, which fall
under the case $n = k\ell$. Additionally, for the case $k = 2$ with $n
> k\ell$, a $(k,\ell)$-subpartition is a partial matching; in their
recent paper, Kamat and Misra \cite{EKRpmatch} presented the
corresponding EKR theorems for these objects. They incorporate a very
nice Katona-style proof, but interestingly, it does not appear that
the Katona method would work very well for $(k,\ell)$-subpartitions
(it seems that this proof would require an additional lower bound on
$n$). The goal of this work is to complete the work done in
both~\cite{MnM} and \cite{EKRpmatch} by showing that an EKR-type
theorem holds for subpartitions. In this paper, we specifically do not
consider the case where $k=2$ (as this is done in Kamat and Misra's
work). In Meagher and Moura~\cite{MnM}, the only difficult case is
$k=2$; it is possible that our counting method will work for the
partial matchings if some of the tricks used in \cite{MnM} are
applied.

\section{Three Technical Lemmas}

We shall require lemmas similar to the Lemma 3 used by Meagher and
Moura in \cite{MnM}---the proofs of which use similar counting
arguments. As we shall see, it is worthwhile to consider the size of a
canonical $t$-intersecting family of $(k,\ell)$-subpartitions, and
find when this is an upper bound for the size of any $t$-intersecting
family of $(k,\ell)$-subpartitions.

Define a \emph{dominating set} for a family of
$(k,\ell)$-subpartitions to be a set of classes, each of size $k$,
that intersects with every $(k,\ell)$-subpartition in the family. For
the intersecting families being investigated here, each
$(k,\ell)$-subpartition in the family is also a dominating
set. In~\cite{MnM}, dominating sets were called \emph{blocking
  sets}. We use the term dominating set here because if the classes in
the $(k,\ell)$-subpartitions (the $k$-sets) are considered to be
vertices, then each $(k,\ell)$-subpartition can be thought of as an
edge in an $\ell$-uniform hypergraph on these vertices. As a result, a
family of $(k,\ell)$-subpartitions is a hypergraph, and our definition
of a dominating set for a family of $(k,\ell)$-subpartitions matches
the definition of a dominating set for a hypergraph.

\begin{lemma}
  Let $n, k, \ell$ be positive integers with $n \geq k\ell$, $\ell \geq 2$ and let
  $\mathcal{P} \subseteq U_{\ell,k}^{n}$ be an intersecting family of
  $(k,\ell)$-subpartitions. Assume that there does not exist a $k$-set
  that occurs as a class in every $(k,\ell)$-subpartition in
  $\mathcal{P}$. Then
\begin{equation} \label{noncanon1}
|\mathcal{P}| \leq \ell^2U(n-2k,\ell-2,k).
\end{equation}
\end{lemma}
\begin{proof}
  Let $\{P_1,\dots,P_\ell\}$ be a $(k,\ell)$-subpartition in
  $\mathcal{P}$ and, for $i \in \{1,\dots,\ell\}$, let
  $\mathcal{P}_i$ be the set of all $(k,\ell)$-subpartitions in
  $\mathcal{P}$ that contain the class $P_i$, but none of
  $P_1,\dots,P_{i-1}$. By assumption,
  $P_i$ does not appear in every $(k,\ell)$-subpartition in
  $\mathcal{P}$, so there exists some $(k,\ell)$-subpartition $Q$ that
  does not contain $P_i$. The subpartitions in $\mathcal{P}_i$ and $Q$ must be
  intersecting, so each member of $\mathcal{P}_i$ must contain $P_i$
  as well as one of the $\ell$ classes from $Q$. Thus, we can bound the size of
  $\mathcal{P}_i$ by
\[
|\mathcal{P}_i| \leq \ell U(n-2k,\ell-2,k).
\]
Further, since $\{P_1,\dots,P_\ell\}$ is a dominating set for the family of $(k,\ell)$-subpartitions, we have that
\[
\bigcup\limits_{i \in \{1,\dots,\ell\}}\mathcal{P}_i = \mathcal{P}.
\]
It follows that
\[
|\mathcal{P}| \leq \ell |\mathcal{P}_i| \leq \ell^2 U(n-2k,\ell-2,k),
\]
as required.
\end{proof}

Note that Lemma 1 certainly applies for all $n \geq k\ell$; however,
if the size of $n$ is small enough relative to $k$ and $\ell$, then we
can improve our bound on such an intersecting family $\mathcal{P}$.
Note that in the case of $n = k\ell$, we may use the lemma as
considered by Meagher and Moura in~\cite{MnM}.

\begin{lemma}
  Let $n, k, \ell$ be positive integers with $k\ell + 1 \leq n \leq
  k(\ell+1)-1$, $\ell \geq 2$, and let $\mathcal{P} \subseteq U_{\ell,k}^{n}$ be an
  intersecting family of $(k,\ell)$-subpartitions. Assume that there
  does not exist a $k$-set that occurs as a class in every
  $(k,\ell)$-subpartition in $\mathcal{P}$. Then
\begin{equation} \label{noncanon2}
  |\mathcal{P}| \leq \ell(\ell-1)U(n-2k,\ell-2,k).
\end{equation}
\end{lemma}
\begin{proof}
  Under the restriction on the size of $n$, there are at most $\ell-1$
  classes in $Q$ that do not contain an element from $P_i$. The
  remainder of the proof follows similarly.
\end{proof}

We also adapt a similar lemma for the $t$-intersecting case.

\begin{lemma}\label{four}
  Let $n, k, \ell, t$ be positive integers with $1 \leq t \leq
  \ell-1$, and let $\mathcal{P} \subseteq
  U_{\ell,k}^{n}$ be a $t$-intersecting family of
  $(k,\ell)$-subpartitions. Assume that there does not exist a $k$-set
  that occurs as a class in every $(k,\ell)$-subpartition in
  $\mathcal{P}$. Then
\begin{equation} \label{noncanont}
|\mathcal{P}| \leq (\ell-t+1) \binom{\ell}{t}U(n-(t+1)k,\ell-(t+1),k).
\end{equation}
\begin{proof}
  As in the proof of Lemma 1, let $\{P_1,\dots,P_\ell\}$ be a
  $(k,\ell)$-subpartition in $\mathcal{P}$ and, for $i \in
  \{1,\dots,\ell\}$, define the set $\mathcal{P}_i$ similarly.  Note
  that if we order the $\mathcal{P}_i$ sets, then any
  $(k,\ell)$-subpartition in $\mathcal{P}_{i}$ where $i> \ell-t+1$
  must contain at least one of the classes $\{P_1, \dots, P_{\ell-t+1}\}$, since the
  $(k,\ell)$-subpartitions here must be $t$-intersecting with
  $\{P_1,\dots,P_\ell\}$. The class $P_i$
  does not appear in every $(k,\ell)$-subpartition in $\mathcal{P}$,
  so there exists some $(k,\ell)$-subpartition $Q$ that does not
  contain $P_i$. Any $(k,\ell)$-subpartition $P \in \mathcal{P}_i$ must be $t$-intersecting with Q,
  so there are $\binom{\ell}{t}$ ways to choose the $t$ classes from $Q$
  that are also in $P$. Thus, we can bound the size of $\mathcal{P}_i$
  by
\[
|\mathcal{P}_i| \leq \binom{\ell}{t} U(n-(t+1)k,\ell-(t+1),k).
\]
Further, since
\[
\bigcup\limits_{i \in \{1,\dots,\ell-t+1\}}\mathcal{P}_i = \mathcal{P},
\]
it follows that
\[
|\mathcal{P}| \leq (\ell-t+1) \binom{\ell}{t} U(n-(t+1)k,\ell-(t+1),k),
\]
as required.
\end{proof}
\end{lemma}
\section{Proof of Theorem 1}

We can use \eqref{noncanon1} or \eqref{noncanon2}, based on the size
of $n$, and compare these bounds with that of \eqref{canon1}.
Informally, we may think of these as bounds on the size of
\emph{non-canonical} families of $(k,\ell)$-subpartitions. If the size
of the canonical family is larger than these bounds, then we know that
the canonical families are the largest and that equality holds if and
only if the intersecting family is canonical.

\begin{proof}[Proof of Theorem~\ref{one}]
Let $\mathcal{P}$ be a non-canonical family of intersecting $(k,\ell)$-subpartitions. We shall show that
\begin{equation} \label{canoncancel}
|\mathcal{P}| < \frac{1}{\ell-1}\binom{n-k}{k} U(n-2k,\ell-2,k).
\end{equation}
It can be verified from \eqref{canon1} and \eqref{recursive} that the
right-hand side of this equation is the size of a canonical
intersecting family of $(k,\ell)$-subpartitions; thus, proving this
equation proves Theorem~\ref{one}.

{\bf Case 1:} $k\ell+1 \leq n \leq k(\ell+1)-1$

If we bound $n$ as such, then by (\ref{noncanon2}),
\[
 |\mathcal{P}| \leq \ell(\ell-1)U(n-2k,\ell-2,k),
\]
and using \eqref{canoncancel}, we only need to prove that
\begin{equation} \label{nsmall}
\ell(\ell-1)^2 \leq \binom{n-k}{k}.
\end{equation}

Since $n \geq k\ell+1$, and using that $k \geq 3$, then by Pascal's rule:
$$\binom{n-k}{k} \geq \binom{k(\ell-1)+1}{k} \geq \binom{3(\ell-1)+1}{3} = \frac{(3\ell-2)(3\ell-3)(3\ell-4)}{3!}.$$

Thus, \eqref{nsmall} can be reduced to checking the inequality
\[
\ell(\ell-1)^2 \leq \frac{(3\ell-2)(3\ell-3)(3\ell-4)}{3!}.
\]
It can be verified, using the increasing function test, that this holds for all $\ell \geq 2$.
{\bf Case 2:} $n \geq k(\ell+1)$

Similar to the previous case, using \eqref{noncanon1} and \eqref{canoncancel}, we only need to show that
\begin{equation} \label{nlarge}
\ell^2(\ell-1) \leq \binom{n-k}{k}.
\end{equation}

As before, taking $n \geq k(\ell+1)$, $k \geq 3$, and using Pascal's rule, we find
$$\binom{n-k}{k} \geq \binom{k\ell}{k} \geq \binom{3\ell}{3} = \frac{3\ell(3\ell-1)(3\ell-2)}{3!}.$$
So, \eqref{nlarge} can be rewritten as
\[
\ell^2(\ell-1) \leq \frac{3\ell(3\ell-1)(3\ell-2)}{3!},
\]
and we find that this also holds for all $\ell \geq 2$.

Thus, (\ref{canoncancel}) holds for all values of $n$, completing the
proof of Theorem 1. \end{proof}

\section{Proof of Theorem 2}

Theorem 2 incorporates the $t$-intersection property, proving a more
general EKR-type theorem for $(k,\ell)$-subpartitions.  Here, the
precise lower bound on $n$ for determining when only the canonical
families are the largest is unknown---but we shall see that if $k \geq
t+2$, then it suffices to take $n \geq k(\ell+t)$ (though this bound
is not optimal).

\begin{proof}[Proof of Theorem 2]
The size of a canonical $t$-intersecting family of $(k,\ell)$-subpartitions, using \eqref{canont} and \eqref{recursive}, is
\begin{equation} \label{canontcancel}
U(n-tk,\ell-t,k) = \frac{1}{\ell-t} \binom{n-tk}{k} U(n-(t+1)k,\ell-(t+1),k).
\end{equation}

As before, let $\mathcal{P}$ be a non-canonical family of
$t$-intersecting $(k,\ell)$-subpartitions. If there is a class that is
contained in every $(k,\ell)$-subpartition of $\mathcal{P}$, then it
can be removed from every such subpartition in $\mathcal{P}$. This
does not change the size of the family, but reduces $n$ by $k$ and
each of $\ell$ and $t$ by $1$. Now we only need to show that this new
family is smaller than the canonical $(t-1)$-intersecting family of
$(k,\ell-1)$-subpartitions from $[n-k]$ (the size of which is equal to
$U(n-(t-1)k,\ell-(t-1),k$). As such, we may assume that there are no
classes common to every $(k,\ell)$-subpartition in $\mathcal{P}$, and
we can apply \eqref{noncanont}.

To prove this theorem, we need to prove that for $n$ sufficiently large,
\begin{equation} \label{tcomparison}
  (\ell-t+1)(\ell-t)\binom{\ell}{t} < \binom{n-tk}{k}.
\end{equation} 
Clearly, this inequality is strict if $n$ is
sufficiently large relative to $t$, $\ell$ and $k$.
\end{proof}

Consider the case where $k \geq t+2$. If $n \geq k(\ell+t)$, then
(\ref{tcomparison}) holds when
\[
  (\ell-t+1)(\ell-t)\binom{\ell}{t} \leq \binom{\ell k}{k}.
\]
Since $k \geq t+2$, we have that
\[
\binom{\ell k}{k} = \left( \frac{\ell k}{k}\right) \left(\frac{\ell k - 1}{k-2}\right)\binom{\ell k -2}{k-2}
   > (\ell -t + 1) (\ell-t) \binom{\ell}{t},
\]
so (\ref{tcomparison}) holds indeed. We do not attempt to find the function $n_0(k,\ell,t)$ that produces the exact lower bound on $n$, but such a lower bound is needed, as shown by the example in
\cite[Section 5]{MnM}.

\section{Extensions}

There are versions of the EKR theorem for many different objects. In
this final section, we shall outline how this method can be generalized to these different objects.

In general, when considering an EKR-type theorem, there is a set of
objects with some notion of intersection. We shall consider the case
when each object is comprised of $k$ \textsl{atoms}, and two objects
are intersecting if they both contain a common atom. If the objects
are $k$-sets, then the atoms are the elements from $\{1,\dots,n\}$,
and each $k$-set contains exactly $k$ atoms. For matchings, the atoms
are edges from the complete graph on $2n$ vertices, and a $k$-matching has $k$ atoms. In this
paradigm, if the largest set of intersecting objects is the set of all
the objects that contain a fixed atom, then an EKR-type theorem holds.

We can apply the method in this paper to this more general
situation. Assume we have a set of objects and that each object
contains exactly $k$ distinct atoms from a set of $n$ atoms (there may be many
additional rules on which sets of atoms constitute an object). Let
$P(n,k)$ be the total number of objects, $P(n-1,k-1)$ the number of
objects that contain a fixed atom, and $P(n-2,k-2)$ the number of
objects that contain two fixed atoms.

Using the same argument as in this paper, if for some type of object
(as above) we have
\[
k^2P(n-2,k-2) < P(n-1,k-1),
\]
then an EKR-type theorem holds for these objects. It is very
interesting to note that if the ratio between $P(n-1,k-1)$ and
$P(n-2,k-2)$ is sufficiently large, then an EKR-type
theorem holds.

For example, this can be applied to $k$-sets. In this case, the equation is
\[
k^2 \binom{n-2}{k-2} < \binom{n-1}{k-1},
\]
which holds if and only if
\[
k^2(k-1) + 1 < n.
\]
This proves the standard EKR theorem, but with a very bad lower bound on $n$.

For a second example, consider length-$n$ integer sequences with entries from
$\{0,1,\dots, q-1\}$. In this case the atoms are ordered pairs $(i,a)$
where the entry in position $i$ of the sequence is $a$. Two sequences
``intersect'' if they have the same entry in the same position. Each
sequence contains exactly $n$ atoms, so in this case $k=n$. The values
of $P(n-1,n-1)$ and $P(n-2,n-2)$ are $q^{n-1}$ and $q^{n-2}$,
respectively.  Thus the EKR-type theorem for integer sequences holds
if $n^2 q^{n-2} < q^{n-1}$, or equivalently if $n^2 < q$. Once again
we have a simple proof of the an EKR-type theorem, but with an
unnecessary bound on $n$.


For a final example consider the blocks in a $t$-$(n,m,\lambda)$
design.  The blocks are $m$-sets so the are $t$-intersecting if they
contain a common set of $t$-elements. It is straight-forward to
calculate the number of a blocks the contain any $s$-set where
$s \leq t$ is 
\[
\lambda \frac{\binom{n-s}{t-s}}{\binom{m-s}{t-s}}.
\]
Thus we have that the EKR theorem holds for intersecting blocks in a $t$-$(n,m,\lambda)$ if
 \[
 m^2 \frac{\lambda \binom{n}{t}\binom{m}{2}}{\binom{m}{t}\binom{n}{2}} 
  \leq  \frac{\lambda \binom{n}{t}\binom{m}{1}}{\binom{m}{t}\binom{n}{1}}
\]
which reduces to
\[
 m^3 -m^2 +1 <n.
\]
This is exactly the bound found by Rands~\cite{EKRblock}.  Moreover,
this method can be applied to $s$-intersecting blocks in a design; again we get
the same bound as in \cite{EKRblock}.



\newpage
\begin{thebibliography}{10}

\bibitem{MR1429238}
R.~Ahlswede and L.~H. Khachatrian.
\newblock The complete intersection theorem for systems of finite sets.
\newblock {\em European J. Combin.}, 18(2):125--136, 1997.

\bibitem{EKRinj}
F.~Brunk and S.~Huczynska.
\newblock Some {E}rd{\H o}s-{K}o-{R}ado theorems for injections.
\newblock {\em European J. Combin.}, 31(3):839--860, 2010

\bibitem{EKRperm}
P.~Cameron and C.~Y. Ku.
\newblock Intersecting families of permutations.
\newblock {\em European J. Combin.}, 24(7):881--890, 2003.

\bibitem{MR2246267}
C.~Colbourn and J.~Dinitz, editors.
\newblock {\em Handbook of {C}ombinatorial {D}esigns}.
\newblock Discrete Mathematics and its Applications. Chapman \& Hall/CRC, Boca
  Raton, FL, second edition, 2007.


\bibitem {MR905276}
P.~Erd{\H{o}}s,
\newblock My joint work with {R}ichard {R}ado.
\newblock Surveys in combinatorics, London Math. Soc. Lecture Note Ser. 123:53--80, 1987.

\bibitem{TheOriginalEKR}
P.~{E}rd{\H o}s, Chao Ko, and R.~Rado.
\newblock Intersection theorems for systems of finite sets.
\newblock {\em The Quarterly Journal of Mathematics}, 12(1):313--320, 1961.

\bibitem{EKRvspace}
P.~Frankl and R.~M. Wilson.
\newblock The {E}rd{\H o}s-{K}o-{R}ado theorem for vector spaces.
\newblock {\em J. Combin. Theory Ser. A}, 43(2):228--236, 1986.

\bibitem{EKRpmatch}
V.~Kamat and N.~Misra.
\newblock An {E}rd{\H {o}}s-{K}o-{R}ado theorem for matchings in the complete
  graph.
\newblock 2013.

\bibitem{EKRpperm}
C.~Y. Ku and I.~Leader.
\newblock An {E}rd{\H o}s-{K}o-{R}ado theorem for partial permutations.
\newblock {\em Discrete Math.}, 306(1):74--86, 2006.

\bibitem{EKRpart}
C.~Y. Ku and D.~Renshaw.
\newblock Erd{\H o}s-{K}o-{R}ado theorems for permutations and set partitions.
\newblock {\em J. Combin. Theory Ser. A}, 115(6):1008--1020, 2008.

\bibitem{MnM}
K.~Meagher and L.~Moura.
\newblock Erd{\H{o}}s-{K}o-{R}ado theorems for uniform set-partition systems.
\newblock {\em Electron J. Combin.}, 12(1):Research Paper 40, 12 pp.
  (electronic), 2005.

\bibitem{EKRblock}
B.~M.~I. Rands.
\newblock An extension of the {E}rd{\H o}s, {K}o, {R}ado theorem to
{$t$}-designs.
\newblock {\em J. Combin. Theory Ser. A}, 32(3):391--395, 1982.

\bibitem{MR0771733}
R.~M. Wilson.
\newblock The exact bound in the {E}rd{\H o}s-{K}o-{R}ado theorem.
\newblock {\em Combinatorica}, 4(2-3):247--257, 1984.

\end{thebibliography}
\end{document}